\documentclass[12pt]{article}
      \usepackage{setspace}     
     \usepackage[pdftex]{graphicx}
     \usepackage{sidecap}
\usepackage{fancyhdr, amsmath, amssymb, amsfonts, url, dsfont, mathrsfs, graphicx, epsfig,  amsthm,mathrsfs}
\usepackage{tikz}
\usepackage{wrapfig}
\usepackage{framed}
\usepackage{listings}

\usepackage{pgf}
\usepackage{tikz}
\usetikzlibrary{arrows,automata}

\usepackage{dsfont}
\usepackage{float}
\usepackage{verbatim} 
\usepackage[latin1]{inputenc}
\usepackage{cite}
\usepackage{graphicx}
\usepackage{vmargin} 
\usepackage{ifpdf} 
\usepackage{url} 
\usepackage{fancyhdr}
\usepackage{lmodern}
\usepackage{moreverb}
\usepackage{enumitem}
\usepackage{amsmath}
\usepackage{pgf}
\usepackage{tikz}
\usepackage{graphicx}
\usetikzlibrary{arrows}
\usetikzlibrary{shapes,snakes,calendar,matrix,backgrounds,folding}

\usepackage{color}
\definecolor{rltred}{rgb}{0.75,0,0}
\definecolor{rltgreen}{rgb}{0,0.5,0}
\definecolor{rltblue}{rgb}{0,0,0.75}

\usepackage{thumbpdf}
\usepackage[pdftex,colorlinks=true,urlcolor=rltred, filecolor=rltgreen, linkcolor=rltblue]{hyperref}

      \usepackage{caption}
      \usepackage{subcaption}
      \usepackage{breqn}	
      \theoremstyle{plain}
      \newtheorem{theorem}{Theorem}[section]
      \newtheorem{lemma}[theorem]{Lemma}
      \newtheorem{corollary}[theorem]{Corollary}
      \newtheorem{remark}[theorem]{Remark}
      \newtheorem{proposition}[theorem]{Proposition}

      \newtheorem{definition}[theorem]{Definition}
      \newtheorem{example}[theorem]{Example}

      \newcommand{\R}{{\mathbb R}}
      \newcommand{\A}{\mathcal{A}}

      \newcommand{\B}{\mathcal{B}}
      \newcommand{\Z}{\mathbb{Z}}
      \newcommand{\N}{\mathbb{N}}

      \newcommand{\C}{\mathbb{C}}

      \newcommand{\E}{\mathbb{E}}      
      \newcommand{\MM}{\mathcal{M}}

      \newcommand{\1}{\mathds{1}}   
      
\begin{document}

\title{Entry time statistics to different shrinking sets}

\author{Italo Cipriano}
\date{April 2016}

\maketitle

\begin{abstract}
We consider $\psi$-mixing dynamical systems $(\mathcal{X},T,\B_{\mathcal{X}},\mu)$ and we find conditions on families of sets $\{\mathcal{U}_n\subset \mathcal{X}:n\in\N\}$ so that $\mu(\mathcal{U}_n)\tau_n$ tends in law to an exponential random variable, where $\tau_n$ is the entry time to $\mathcal{U}_n.$
\end{abstract}

 \section{Introduction}\label{ERTS_intro}

We develop further results about higher order entry times. That is the rate at which points enter to small sets.
Consider a measure preserving dynamical system $(\mathcal{X},\B_{\mathcal{X}},\mu,T)$ where $\mu$ is a finite, invariant and ergodic measure together with a sequence of Borel sets $\{\mathcal{U}_n\}=\{\mathcal{U}_n\}_{n\in\N}$ with $\mathcal{U}_n\subset \mathcal{X}, \mu(U_n)>0$ and for which the sequence $\{\mathcal{U}_n\}$ shrinks to a point. Under an appropriate mixing assumption, one can generally show that for $$\tau_{n}(x)=\tau_{\mathcal{U}_n}(x):=\inf\{k\geq 1: T^k(x)\in \mathcal{U}_n\},$$
the sequence of random variables $X_n:=\mu(\mathcal{U}_n)\tau_{n},$ called (rescaled) \emph{entry time} or \emph{first hitting time}, converges in law to an exponential random variable. The first paper related to this result is \cite{Doeblin40} for continued fractions. More recently such convergence results have been obtained for examples in which $U_n$ are balls or cylinders shrinking to a point: for continuous time Markov chains see \cite{AandB1992,AandB1993}, for expanding maps of the interval see \cite{GS93,GS95,CEbook}, for general $\phi$-mixing processes (consequently also for $\psi$-mixing processes) see \cite{GS97}, for $\alpha$-mixing processes see \cite{Ab2001,AbSa2011}, for Axiom A diffeomorphisms see \cite{Hirata,MH95}, for Gibbs measures on shift spaces see \cite{Pitskel}, for unimodal maps see \cite{Bruin_returntime},
for partially hyperbolic systems see \cite{Dolgopyat2004,ChCo2013}. Some extensions of the classic Poisson limit theorem can be found in \cite{MR3101254} and further related reviews in \cite{AbGal2001,MR1776758,MR3170620}. 
It is not clear however that the limit distribution of the sequence of random variables $X_n$ is still valid if we remove the condition that $\{\mathcal{U}_n\}$ are shrinking to a single point.\\

In this paper we consider the problem of finding conditions on the pair $(\{\mathcal U_n\},\mu)$ so that $X_n$ converges in law to an exponential random variable when  $\mu(\mathcal{U}_n)\to 0$ as $n\to\infty,$ but the sequence $\{\mathcal{U}_n\}$ does not necessarily shrink to a single point (or finitely many points). The conditions we impose on $\mathcal U_n$ and $\mu$ are similar to the used in \cite{Collet2006,MR2165587,Haydn_2014}. In many cases these conditions are easy to verify, this allows us to find new systems satisfying exponential limit distribution.\\

In our main theorem we obtain the convergence of $X_n$ to an exponential random variable for general families of sets $\{\mathcal{U}_n\}$ under some conditions that depend on the return time to $\mathcal{U}_n$ defined by $$\eta_n:=\inf\{\tau_n(x):x\in  \mathcal{U}_n\}.$$ Indeed, let $(\mathcal X,\sigma)$ be a topologically mixing subshift of finite type with $$\mathcal X\subset \prod_{-\infty}^{\infty}\{1,\ldots,a\}=\{1,\ldots,a\}^{\Z}.$$ We consider the pair $(\{\mathcal U_n\},\mu),$ where $\mathcal{U}_n\subset \mathcal X$ is in the sigma-algebra generated by $\prod_{-n}^{n-1} \{1,\ldots,a\}$ and $\mu=\mu_{\text{Parry}}$ is the probability measure of maximum entropy or Parry measure.

\begin{theorem}[Main Theorem]\label{teo:29062015}
 The sequence of random variables $X_n$ converges in law to an exponential random variable if
\begin{enumerate}
\item \label{c1_250216} $\sigma^{-(n+1)}\mathcal{U}_{n+1} \subset \sigma^{-n}\mathcal{U}_n$ for every $n>0$ and $n\mu\left(\mathcal{U}_{\left\lfloor \eta_n/2 \right\rfloor} \right)\to 0$ as $n\to\infty,$
\item \label{c2_250216} the return times are given by $\eta_n=n+k(n)+1,$ where $k:\N\to\N$ is a non decreasing function and $n\mu\left(\mathcal{U}_{k(n)}\right)\to 0$ as $n\to\infty,$ or
\item \label{c3_250216} there exists a sequence $\{\mathcal{V}_n\}$ with $n\mu(\mathcal{V}_{n})\to 0$ as $n\to\infty,$ where $\mathcal{V}_n\supset \mathcal{U}_n$ is a set in the sigma-algebra generated by $\prod_{-n}^{-n+ \left\lfloor  \eta_n/2 \right\rfloor } \{1,\ldots,a\}$ or $\prod_{n- \left\lfloor  \eta_n/2 \right\rfloor }^{n-1} \{1,\ldots,a\}$ for every $n\geq 1.$ 
\end{enumerate}
\end{theorem}

An application of our method is to the study of Gibbs measures and sets which do not necessarily shrink to single points. Indeed, suppose that we have two subshifts of finite type $\mathcal{X}$ and $\mathcal{Y}\subset\mathcal{X},$ and the family of sets $\{\mathcal{U}_n\},\mathcal{U}_n \subset \mathcal X$ satisfies $\cap_{n=1}^{\infty}\mathcal{U}_n=\{xy:y\in\mathcal{Y}\},$ where $x$ is a forbidden sequence in $\mathcal{Y}.$ We prove that 
for $\mu$ a Gibbs measure of H\"older potential, and under suitable conditions, that depend on $\mu$ and the topological entropy of $(\mathcal{X},\sigma)$ and $(\mathcal{Y},\sigma),$ the sequence $X_n$ converges in law to an exponential random variable. This is closely related to the results in \cite{Collet2006} but our proofs are motivated by \cite{CEbook}. Recently, a necessary and sufficient condition for $X_n$ to converge in law to an exponential random variable when $\mu$ is an ergodic probability measure is established in Theorem 1 in \cite{Zhang}, however, to check this condition requires not straightforward estimations.

\section{Background}
  
 We will formally introduce the definition of subshift of finite type. Let $A$ denote an irreducible and aperiodic $a\times a$ matrix of zeros and ones ($a\geq 2$), i.e. there exists $d\in\N$ for which $A^{d}>0$ (all coordinates of $A^d$ are strictly positive). We call the matrix $A$ transition matrix. We define the subshift of finite type $\mathcal{X}=\mathcal{X}_A \subset \{ 1,\ldots,a\}^{\Z}$ such that
  $$\mathcal{X}:=\{(x_n)_{n=-\infty}^{\infty}:A(x_n,x_{n+1})=1\mbox{ for all  }n\in \Z\}.$$ On $\mathcal{X},$ the shift $\sigma:\mathcal{X}\to \mathcal{X}$ is defined by $\sigma(x)_{n}=x_{n+1}$ for all $n\in \Z.$ 
       For $x\in\mathcal{X}$ and $n,m\in \Z, m>n$ we define the cylinder $$ [x]_n^m:=\{y\in\mathcal{X}:y_i=x_i,\mbox{ for } i\in \{ n,\ldots, m-1\} \}$$
and also denote $x_{[n,m)}:=x_nx_{n+1}\ldots x_{m-1}=(x_k)_{k=n}^{m-1},$ which corresponds to the concatenation of $m-n$ elements in $\{ 1,\ldots, a\}.$ We denote by $\xi_{n}^{m}$ the set of all the cylinders $[x]_n^m$ with $x\in\mathcal{X}.$ In the particular, when $n=0$ we denote $[x]_n^m=: [x]_m$ and $\xi_{n}^{m}=:\xi_{m}.$  Denote by $\MM_{\sigma}$ the space of $\sigma$ invariant probability measure on $\mathcal{X},$ that is the space of probability measures $\mu$ on $\B_{\mathcal{X}}$ such that $\mu(\A)=\mu(\sigma^{-1}\A)$ for every measurable set $\A.$\\
For $\theta \in (0,1),$ we consider the metric $d_{\theta}(x,y)=\theta^{m}$ where $m=\inf\{n\geq 0:x_{n}\neq y_n\mbox{ or }x_{-n}\neq y_{-n}\}$ and $d(x,x)=0$ for every $x\in\mathcal{X},$ we have in particular that $(\mathcal{X},d_{\theta})$ is a complete metric space. We say that $f:\mathcal{X}\to\R$ is continuous if it is continuous with respect to $d_{\theta}.$ Given $f:\mathcal{X}\to\R$ continuous and $m \in \N$ define 
 \[
 V_m(f):= \sup\{|f(x)-f(y)|:x,y\in\mathcal{X} \mbox{ and } x_i=y_i \mbox{ }\forall i\in \{ -m,\ldots, m \}\},
 \]
and the Lipschitz semi-norm
\[
|f|_{\theta}:=\sup\left\{ \frac{V_m(f)}{\theta^m}:m \in\N\right\}.
\]    
Since constant functions all have Lipschitz semi-norm equal to zero, one needs to define the norm on the space of Lipschitz functions by 
\[
\left\Vert f\right\Vert_{\theta}:= |f|_{\theta}+\| f \|,
\]
where $\| f \|:=\sup_{x\in \mathcal{X}} \{|f(x)|\}.$
The space of continuous functions with finite Lipschitz norm is called the space of Lipschitz functions (or $\theta$-Lipschitz functions). Recall that a continuous function is $\alpha$-H\"older for $d_{\theta}$ if and only if it is Lipschitz for $d_{\theta^{\alpha}}.$

\begin{definition} Given a subshift of finite type $(\mathcal{X},\sigma)$ with incidence matrix $A$ of size $a\times a,$ the topological entropy $h$ of $\mathcal{X}$ is defined as the logarithm of the first eigenvalue of $A,$ i.e., for $\lambda:=\max\{|\gamma|:Av=\gamma v,\gamma\in\C,v\in \C^{a}\},$ $h:=\log \lambda.$ On the space of invariant probability measures $\MM_{\sigma}$ there exists one and only one measure $\mu$ (see \cite{lindmarcus}) such that 
\begin{equation}\label{eq1_16_09_2015}
h=-\lim_{n\to\infty}\frac{1}{n}\sum_{\mathcal W\in\xi_n}\mu(\mathcal W)\log \mu (\mathcal W).
\end{equation}
The measure $\mu\in\MM_{\sigma}$ that achieves the supremum in (\ref{eq1_16_09_2015}) is the called measure of maximum entropy or Parry measure, that we denote by $\mu_{\text{Parry}}.$  
\end{definition}

Suppose that we have a subshift of finite type $(\mathcal X,\sigma)$ and an invariant probability measure $\mu$ on $\B_{\mathcal{X}}.$ We will need the following mixing condition on the measure $\mu.$

\begin{definition}
The measure preserving dynamical system $(\mathcal{X},\B_{\mathcal{X}},\mu,\sigma)$ is $\psi$-mixing if for $\mathcal U$ in the $\sigma$-algebra generated by $\{ 1,\ldots, a\}^{n}$ and $\mathcal V$ in the $\sigma$-algebra generated by $\{ 1,\ldots,a\}^{*}:=\cup_{n\in\N}\{ 1,\ldots,a\}^n$ as the `gap' $\triangle\to\infty:$ 
$$
\sup_n\sup_{\mathcal U,\mathcal V} \left| \frac{\mu( \mathcal{U}\cap \sigma^{-\triangle-n}\mathcal{V})}{\mu(\mathcal{U})\mu(\mathcal{V})} -1 \right|=\psi_\triangle\to 0. 
$$
\end{definition}

Let us recall some basic definitions on probability theory that we will use.

\begin{definition}
Let $(\Omega,\mathcal{F},\mathbb{P})$ be a probability space, i.e. $\Omega$ is a set, $\mathcal{F}$ a sigma-algebra on $\Omega$ and $\mathbb{P}$ is a probability measure on $(\Omega,\mathbb{F}).$ The expectation with respect to $\mathbb{P}$ of a measurable function $f: \Omega \to \R$ is defined by
$$
\E(f)=\E_{\mathbb{P}}(f):=\int_{\Omega} f d\mathbb{P}.
$$
\end{definition}

\begin{definition}
A random variable $X:\Omega \to\R$ on a probability space $(\Omega,\mathcal{F},\mathbb{P})$ is said to be an exponential random variable with parameter $\lambda$ if it has cumulative distribution
$$F_X(t):=\mathbb{P}\{w\in \Omega: X(w)\leq t\}= \begin{cases} 
1-e^{-\lambda t} &\mbox{if } t\geq 0, \\
0  &\mbox{if } t <0.\end{cases}
$$
\end{definition}

\begin{definition}
A sequence of random variables $\{X_n\}$ on a probability space $(\Omega,\mathcal{F},\mathbb{P})$ is said to converge in law to an exponential random variable of parameter $1$ if for every $t\in\R,$ its cumulative distribution converges to the cumulative distribution $F_X(t),$ of an exponential random variable $X$ of parameter $\lambda=1.$ In other words, 
\begin{equation}\label{NO_trick_one}
\lim_{n\to\infty}|\mathbb{P}\{w\in\Omega: X_n(x)\geq t\}-e^{-t}|=0 \mbox{ for every }t>0.
\end{equation}
\end{definition}

We introduce the definition of some particular dynamical systems together with a sequence of sets whose measure converges to zero and a sequence of random variables built from the entry times to these sets.

\begin{definition}[$M$-systems]
We define a $M$-system as any system $$(\mathcal{X},\B_{\mathcal{X}},\mu,T,\{\mathcal{U}_n\},\{X_n\})$$ where
\begin{enumerate}
\item the system $(\mathcal{X},\B_{\mathcal{X}},\mu,T)$ is a measure preserving dynamical system
and $\mu$ is an ergodic probability measure;
\item the sequence  
$\{\mathcal{U}_n\}$ is a sequence of Borel sets $\mathcal{U}_n \subset \mathcal X$ such that $\mu(\mathcal{U}_n)>0$ and $\mu(\mathcal{U}_n)\to 0;$ and
\item the sequence $\{X_n\}$ is a sequence of random variables $X_n:\mathcal X\to\R$ defined by $X_n(x):=\mu(\mathcal U_n)\tau_n(x).$ 
\end{enumerate} 
\end{definition}

A particularly useful $M$-system for us is the case that the measure is $\psi$-mixing.

\begin{definition}[$M_a$-systems]
We define a $M_a$-system as any $M$-system $$(\mathcal{X},\B_{\mathcal{X}},\mu,\sigma,\{\mathcal{U}_n\},\{X_n\})$$ where $(\mathcal X,\sigma)$ is a topologically mixing subshift of finite type with $\mathcal X\subset \{1,\ldots,a\}^{\Z}$ for some integer $a>1$ and the measure preserving dynamical system $(\mathcal{X},\B_{\mathcal{X}},\mu,\sigma)$ is $\psi$-mixing where the sequence $\{\psi_{\triangle}\}_{\triangle\in\N}$ is bounded.
\end{definition}

We investigate the convergence in law of $X_n$ in a $M_a$-system. A useful trick that we learned from \cite{CEbook} is to consider
\begin{equation}\label{collet_trick_one}
\lim_{n\to\infty} \left| \mu\left\{x\in\mathcal{X}: \tau _n (x)> \left\lfloor t/\mu(\mathcal U_n) \right\rfloor\right\}-(1-\mu(\mathcal U_n))^{\left\lfloor t/\mu(\mathcal U_n) \right\rfloor
} \right|=0 \mbox{ for every }t>0,
 \end{equation}
 that is equivalent to (\ref{NO_trick_one}).
 
\section{Intermediate results and examples}
  
  We present an important proposition, indeed a few improvements of it will prove Theorem \ref{teo:29062015}.

\begin{proposition}[Main proposition]\label{prop1_10032015}
Let $(\mathcal{X},\B_{\mathcal{X}},\mu,\sigma,\{\mathcal{U}_n\},\{X_n\})$ be a $M_a$-system where $\mathcal{U}_n$ is in the sigma algebra generated by $\prod_{0}^{n-1}\{ 1 ,\ldots, a\}$ for every $n\geq 1.$ If 
\begin{equation}\label{cond10032015:1}
\eta_n=n 
\end{equation} 
and 
\begin{equation}\label{cond10032015:2}
n \mu(\mathcal{U}_n)\to 0 \mbox{ as }n\to\infty. 
\end{equation} 
Then the sequence of random variables $X_n$ converges in law to an exponential random variable of parameter $1.$
\end{proposition}

Notice that the condition $\eta_n=n$ is equivalent to $\eta_n\geq n,$ because of the definition of the sequence of sets $\{\mathcal U_n\}.$ As an application of the proposition we can give the following example, where $\A\sqcup \B$ denotes the disjoint union of the sets $\A$ and $\B.$ 

\begin{example}
Suppose that $\mathcal{X}=\{0,1,2\}^{\Z}$ and 
$$
\mathcal{U}_n=\sqcup_{x_1,\ldots,x_{n-1}\in \{1,2\} } [0, x_1, x_2,\ldots,x_{n-1}]_n.
$$
 Let $\mu$ be a Bernoulli probability measure on $\mathcal{X}$ defined by a probability vector $(p_0,p_1,p_2)\in (0,1)^3.$ Then  $(\mathcal{X},\B_{\mathcal{X}},\mu,\sigma,\{\mathcal{U}_n\},\{X_n\})$ is a $M_a$-system and (\ref{cond10032015:1}), (\ref{cond10032015:2}) are satisfied, therefore Proposition \ref{prop1_10032015} applies.
\end{example}

Another application is to Gibbs measures.
 
\begin{definition}[Gibbs measures]\label{Gibbs_measure}
Given a subshift of finite type $\mathcal X,$ we say that a probability measure $\mu$ on $\B_{\mathcal{X}}$ is a Gibbs measure (or an equilibrium state) of H\"older potential $\phi:\mathcal X\to \R $ if there is $c_1,c_2>0$ and $P\geq 0$ such that 
$$
c_1\leq \frac{\mu([x]_m)}{\exp\left(-Pm + S^{\sigma}_m\phi(x)\right)}\leq c_2
$$
for every $x\in \mathcal X$ and $m\geq 0,$ where $S^{\sigma}_m\phi(x):=\sum_{k=0}^{m-1}\phi(\sigma^k x).$
\end{definition} 

\begin{remark} Gibbs measures of H\"older potential are $\psi$-mixing, see \cite{Rufus}.
\end{remark}

Before giving an example in the case of Gibbs measures let us mention that Proposition \ref{prop1_10032015} gives a condition that depends on the pressure.

\begin{remark} \label{rem290615}
Let $(\mathcal X,\sigma)$ be a subshift of finite type and $\mathcal U_n= [x^1]_n\sqcup \cdots \sqcup[x^{m_n}]_n,$ where $x^1,\ldots,x^{m_n}\in \mathcal{X},$ for every $n\geq 1.$ If $\mu$ is a Gibbs measure on $\mathcal X$ of H\"older potential $\phi:\mathcal X\to\R,$ then
\begin{equation}\label{eq12_feb_16}
0 \leq \mu(\mathcal U_n)\leq m_n\exp\left(-Pn+\|S_n^{\sigma}\phi\| \right),
\end{equation}
for some constants $c>0.$ In particular, we deduce that the hypothesis
$$
nm_n\exp\left(-Pn+\|S_n^{\sigma}\phi\| \right) \to 0 \mbox{ as }n\to\infty 
$$
is enough in order to satisfy the hypothesis (\ref{cond10032015:2}).
\end{remark}

\begin{example}
Let $A$ be an irreducible and aperiodic matrix with entries $0$ and $1$ of size $a\times a$ with $ a>2$ and let $B$ be a submatrix of $A$ of size $b\times b$ with $b\in \{  2,\ldots,a-1 \} .$ Denote by $\lambda_A$ and $\lambda_B$ the Perron eigenvalues of $A$ and $B,$ respectively. Consider the subshift of finite type  $\mathcal{X}_A\subset \{ 1 ,\ldots, a\}^{\Z}$ and $\mathcal{X}_B\subset \mathcal{X}_A.$ Suppose without lost of generality that $\mathcal{X}_B\subset \{ 1 ,\ldots, b\}^{\Z}.$ Define for $n\in \N,$
$$
\mathcal{U}_n= \sqcup_{x_1,\ldots,x_{n-1}\in \{1,2,\ldots b\} } [a x_1\cdots x_{n-1}]_{n} ,
$$
then the sequence of random variables $X_n:=\mu(\mathcal U_n)\tau_n$ converges in law to an exponential random variable of parameter $1$ for any equilibrium state with H\"older potential $\phi$ such that $3\left\Vert\phi \right\Vert < \lambda_A-\lambda_B.$ As (\ref{cond10032015:2}) is clearly satisfied, then Proposition \ref{prop1_10032015} applies.
\end{example}

We extend Proposition \ref{prop1_10032015} in order to allow short return times. In this case we do need to care about considering shrinking sets, unlike in previous construction.

\begin{definition}
We say that a sequence of sets $\{\mathcal{U}_n\}$ is a sequence of shrinking sets if  $\mathcal{U}_n\supset \mathcal{U}_{n+1}$ for every $n\in\N.$
\end{definition}

We can now state our first corollary.

\begin{corollary}[First corollary]\label{prop2_10032015}Suppose that $(\mathcal{X},\B_{\mathcal{X}},\mu,\sigma,\{\mathcal{U}_n\},\{X_n\})$ is a $M_a$-system where $\mathcal U_n$ is a Borel set in the sigma algebra generated by $\prod_{0}^{n-1}\{1,\ldots,a\}$ for every $n\geq 1.$ If the sequence of sets $\{\mathcal U_n\}$ is shrinking and it satisfies 
\begin{equation}\label{cond20032015:3}
n\mu\left(\mathcal{U}_{\left\lfloor \eta_n/2 \right\rfloor}\right) \to 0 \mbox{ as }n\to\infty, 
\end{equation}
then the sequence of random variables $X_n$ converges in law to an exponential random variable of parameter $1.$
\end{corollary}

\begin{example}
Let us choose $x\in \mathcal{X}$ and consider $\{\mathcal{U}_n\}=\{[x]_n:n\in\N\}.$ If $n\mu\left([x]_{\left\lfloor \eta_n/2 \right\rfloor}\right)\to 0,$ then the sequence of random variables $X_n:=\mu([x]_n)\tau_n$ converges in law to an exponential random variable of parameter $1.$
\end{example}

\begin{example}
In particular, for the measure of maximum entropy $\mu_{\text{Parry}}.$ If $\{\mathcal U_n\}$ is a shrinking sequence and $n\mu_{\text{Parry}}(\mathcal U_{\left\lfloor \eta_n/2 \right\rfloor})\to 0,$ then $X_n:=\mu_{\text{Parry}}(\mathcal U_n)\tau_n$ converges in law to an exponential random variable of parameter $1.$ This result is not sharp, but to have certain control on $\eta_n$ is necessary.
\end{example}

\begin{example}
Suppose that $(\mathcal{X},\sigma)$ is the full shift in two symbols, i.e. its transition matrix $A$ is a two by two matrix with $1$ in each coordinate, and suppose that $\mu$ is the uniform probability measure, i.e. $\mu([x]_n)=2^{-n}$ for every $x\in \mathcal X$ and $n\in\N.$ Then $X_:=\mu(\mathcal U_n)\tau_n$ converges in law to an exponential random variable of parameter $1$ for any sequence of sets $\{\mathcal U_n\}$ such that $\eta_n=\log_2(n)+k_n,$ where $\{k_n\}\subset\N$ is a divergent sequence.
\end{example}

\section{Proofs}

We have separated the proofs of our results into three subsections. In the first we introduce the notation that will be used along the section and we prove Proposition \ref{prop1_10032015}. In the second we prove Corollary \ref{prop2_10032015}. Finally, in the third we combine what we did in  the first two subsections in order to prove Theorem \ref{teo:29062015}. 

\subsection{Main proposition}\label{sub_1_25_4_16}

The goal of this subsection is to prove Proposition \ref{prop1_10032015}. In what follows we suppose that $(\mathcal{X},\sigma, \mu)$ is a subshift of finite type where $\mu$ is a
probability measure on $\B_{\mathcal{X}}$ and that we have a sequence $\{\mathcal{U}_n\}$ of Borel sets $\mathcal{U}_n \subset \mathcal X$ such that $\mu(\mathcal{U}_n)\to 0.$ Instead of (\ref{NO_trick_one}) we consider the equivalent condition (\ref{collet_trick_one}). We prove the assertion (\ref{collet_trick_one}) in several steps. The first is to replace (\ref{collet_trick_one}) by a limit involving the sum of $N$ terms.

\begin{lemma}\label{lem_10032015:1}  For $n\in \N,$ $\epsilon_n=\mu (\mathcal{U}_n)$ and $N=[t/\epsilon_n]$ with $t>0$ we have that 
\begin{equation}\label{eq_5_ene_2015_9}
\mu\{\tau _n >N\}-(1-\epsilon_n)^N=\sum_{q=0}^{N-1} (1-\epsilon_n)^{N-q-1}(\mu\{\tau_n>q+1\}-(1-\epsilon_n)\mu\{\tau_n>q\}).
\end{equation}
\end{lemma}

The proof of Lemma \ref{lem_10032015:1} consists in noticing that most of the term in the sum of the right hand side of the equation (\ref{eq_5_ene_2015_9}) cancel when summed with other term, the unique terms that do not cancel are indeed the ones on the left hand side of the equation.

For what follows we require additionally that  $\mu$ is an invariant probability measure on $\B_{\mathcal{X}}.$ For $p,q\in \N_0$ we use the notation $\mu\{\tau_p>q\}$ instead of $\mu\{x\in\mathcal{X}: \tau_p(x)>q\}$ and  for every $n\in \N$ we define $\epsilon_n:=\mu (\mathcal{U}_n)$ and $N:=[t/\epsilon_n],$ where $t>0.$\\

For $q \in \{ 0 ,\ldots, N-1\}$ define $$p_q(n):=(1-\epsilon_n)^{N-q-1}(\mu\{\tau_n>q+1\}-(1-\epsilon_n)\mu\{\tau_n>q\}),$$ $S_1(N):=\sum_{q=0}^{n-1}p_q(n)$ and $S_2(N):=\sum_{q=n}^{N-1}p_q(n).$ To obtain (\ref{collet_trick_one}) we will show with the help of a few lemmas that $$\mu\{\tau _n >N\}-(1-\epsilon_n)^N=S_1(N)+S_2(N)\to 0 \mbox{ as }n \mbox{ tends to infinity}.$$
 The second step of our proof is to bound the term $S_1(N).$ We have the following lemma.

\begin{lemma}\label{lem_10032015:2}
For all $n\in \N,$ we have that
$S_1(N)\leq n\epsilon_n.$
\end{lemma}

The proof is a direct consequence of a useful identity in the next lemma that requires the measure $\mu$ to be invariant. We  denote by $\E$ the expectation with respect to $\mu.$

\begin{lemma}\label{lem_10032015:3}
For all $n\in \N,$  we have that for all $q\in \{ 0 ,\ldots, N-1\}$
\[
\mu\{\tau_n>q+1\}-(1-\epsilon_n)\mu\{\tau_n>q\}=\epsilon_n \mu\{\tau_n>q\}-\mu\{x\in \mathcal{U}_n:\tau_n(x)>q\}.
\]
\end{lemma}

\begin{proof}
The result is direct from the definitions of the sets $\{\tau_n>q+1\}$ and $\{\tau_n>q\}.$ Indeed, we can write the following identities:
$$
\begin{aligned}
&\mu\{\tau_n>q+1\}-(1-\epsilon_n)\mu\{\tau_n>q\}\cr
&=\E\left( \prod_{i=1}^{q+1}\1_{\mathcal{U}_n^c}\circ \sigma^i  \right)- (1-\epsilon_n)\E\left(\prod_{i=1}^{q}\1_{\mathcal{U}_n^c}\circ \sigma^i  \right)\cr
&=\E\left( \prod_{i=1}^{q+1}\1_{\mathcal{U}_n^c}\circ \sigma^i  \right)- (1-\epsilon_n)\E\left( (\1_{\mathcal{U}_n^c}+\1_{\mathcal{U}_n}) \prod_{i=1}^{q}\1_{\mathcal{U}_n^c}\circ \sigma^i  \right)\cr
&=\E\left( \prod_{i=1}^{q+1}\1_{\mathcal{U}_n^c}\circ \sigma^i  \right)- (1-\epsilon_n)\E\left( \1_{\mathcal{U}_n^c} \prod_{i=1}^{q}\1_{\mathcal{U}_n^c}\circ \sigma^i  \right) -(1-\epsilon_n)\E\left( \1_{\mathcal{U}_n} \prod_{i=1}^{q}\1_{\mathcal{U}_n^c}\circ \sigma^i  \right)\cr
&= \E\left( \prod_{i=1}^{q+1}\1_{\mathcal{U}_n^c}\circ \sigma^i  \right)-\E\left( \1_{\mathcal{U}_n^c} \prod_{i=1}^{q}\1_{\mathcal{U}_n^c}\circ \sigma^i  \right)\cr
&\quad +\epsilon_n\E\left( \1_{U^c_n} \prod_{i=1}^{q}\1_{\mathcal{U}_n^c}\circ \sigma^i  \right) -(1-\epsilon_n)\E\left( \1_{\mathcal{U}_n} \prod_{i=1}^{q}\1_{\mathcal{U}_n^c}\circ \sigma^i  \right)\cr
&=\epsilon_n\E\left( (1-\1_{\mathcal{U}_n}) \prod_{i=1}^{q}\1_{\mathcal{U}_n^c}\circ \sigma^i  \right) -(1-\epsilon_n)\E\left( \1_{\mathcal{U}_n} \prod_{i=1}^{q}\1_{\mathcal{U}_n^c}\circ \sigma^i  \right)\cr
&=\epsilon_n\E\left( \prod_{i=1}^{q}\1_{\mathcal{U}_n^c}\circ \sigma^i  \right) -\E\left( \1_{\mathcal{U}_n} \prod_{i=1}^{q}\1_{\mathcal{U}_n^c}\circ \sigma^i  \right).
\end{aligned}
$$
Therefore we have 
\begin{equation}\label{george_number_eq}
\begin{aligned}
&\mu\{\tau_n>q+1\}-(1-\epsilon_n)\mu\{\tau_n>q\}\cr
& \quad =\epsilon_n\E\left( \prod_{i=1}^{q}\1_{\mathcal{U}_n^c}\circ \sigma^i  \right) -\E\left( \1_{\mathcal{U}_n} \prod_{i=1}^{q}\1_{\mathcal{U}_n^c}\circ \sigma^i  \right),
\end{aligned}
\end{equation}

and this is enough to conclude the proof, because $$\epsilon_n\E\left( \prod_{i=1}^{q}\1_{\mathcal{U}_n^c}\circ \sigma^i  \right) -\E\left( \1_{\mathcal{U}_n} \prod_{i=1}^{q}\1_{\mathcal{U}_n^c}\circ \sigma^i  \right) = \epsilon_n \mu\{\tau_n>q\}-\mu\{x\in \mathcal{U}_n:\tau_n(x)>q\}.$$
\end{proof}

In the third step, and most difficult one we bound $S_2(N).$ This step requires additionally that $\mu$ is $\psi$-mixing. Or equivalently, that $(\mathcal{X},\B_{\mathcal{X}},\mu,\sigma,\{\mathcal{U}_n\},\{X_n\})$ is a $M_a$-system. In order to obtain an upper bound for $S_2(N)$ we will state a lemma with some intermediate bounds.

\begin{lemma}\label{lem_10032015:4}
For all $n,k\in \N,$ we have
\begin{equation}\label{eq_same_lemma_1}
\E(\1_{\mathcal{U}_n}\1_{\mathcal{U}_n}\circ \sigma^{n+k})\leq   \epsilon_n^2 (1+ \psi_k),
\end{equation}
and for $q\in\{ 2n+1 ,\ldots, N-1\}$ we have
\begin{equation}\label{eq_same_lemma_2}
\begin{aligned}
\left|\E\left(\1_{\mathcal{U}_n}\prod_{i=2n}^{q}\1_{\mathcal{U}_n^c}\circ \sigma^i\right)-\E\left(\1_{\mathcal{U}_n}\prod_{i=n}^{q}\1_{\mathcal{U}_n^c}\circ \sigma^i\right) \right| \leq   \epsilon_n^2(n+\sum_{i=0}^{n-1}\psi_i),
\end{aligned}
\end{equation}
\begin{equation}\label{eq_same_lemma_3}
\left|\E\left(\prod_{i=2n}^{q}\1_{\mathcal{U}_n^c}\circ \sigma^i\right)-\E\left(\prod_{i=n}^{q}\1_{\mathcal{U}_n^c}\circ \sigma^i\right) \right|\leq n\epsilon_n
\end{equation}
and 
\begin{equation}\label{eq_same_lemma_4}
\left|\epsilon_n\E\left(\prod_{i=2n}^{q}\1_{\mathcal{U}_n^c}\circ \sigma^i\right)-\E\left(\1_{\mathcal{U}_n}\prod_{i=2n}^{q}\1_{\mathcal{U}_n^c}\circ \sigma^i\right) \right|\leq  \epsilon_n \psi_n.
\end{equation}
\end{lemma}

\begin{proof}
Inequality (\ref{eq_same_lemma_1}) requires the $\psi$-mixing condition and inequality (\ref{eq_same_lemma_2}) is a consequence of (\ref{eq_same_lemma_1}).

\item[Proof of (\ref{eq_same_lemma_1})] We can use the $\psi$-mixing condition to conclude that $$\left|\E(\1_{\mathcal{U}_n}\1_{\mathcal{U}_n}\circ\sigma^{n+k})-\epsilon_n^2\right| =\left|\E(\1_{\mathcal{U}_n}\1_{\mathcal{U}_n}\circ\sigma^{n+k})-\mu(\mathcal{U}_n)^2\right| \leq  \epsilon_n^2 \psi_k. $$
Then trivially $$\E(\1_{\mathcal{U}_n}\1_{\mathcal{U}_n}\circ\sigma^{n+k})\leq  \epsilon_n^2 \psi_k+ \epsilon_n^2,$$  
as required.

\item[Proof of (\ref{eq_same_lemma_2})]Let $n,k\in \N$ and $q\in \{  2n+1 ,\ldots, N-1\}.$ We have 
$$
\begin{aligned}
& \left|\E\left(\1_{\mathcal{U}_n}\prod_{i=2n}^{q}\1_{\mathcal{U}_n^c}\circ \sigma^i\right)-\E\left(\1_{\mathcal{U}_n}\prod_{i=n}^{q}\1_{\mathcal{U}_n^c}\circ \sigma^i\right) \right| \cr
&=\E\left(\1_{\mathcal{U}_n} \cdot \prod_{i=2n}^{q}\1_{\mathcal{U}_n^c}\circ \sigma^i \cdot \left( 1-\prod_{i=n}^{2n-1}\1_{\mathcal{U}_n^c}\circ \sigma^i \right) \right) \cr
&\leq \E\left(\1_{\mathcal{U}_n} \cdot \left( 1-\prod_{i=n}^{2n-1}\1_{\mathcal{U}_n^c}\circ \sigma^i \right) \right)\cr
&\leq \sum_{i=n}^{2n-1}\E\left(\1_{\mathcal{U}_n} \1_{\mathcal{U}_n}\circ \sigma^i \right) \leq
 \epsilon_n^2\sum_{i=0}^{n-1}\psi_i+n\epsilon_n^2,
\end{aligned}
$$
where we have used (\ref{eq_same_lemma_1}) in the last inequality.

\item[Proof of (\ref{eq_same_lemma_3})] Let $n,k\in \N$ and $q\in \{  2n+1,\ldots, N-1\}.$ We have 
$$
\begin{aligned} 
&\left|\E\left(\prod_{i=2n}^{q}\1_{\mathcal{U}_n^c}\circ \sigma^i\right)-\E\left(\prod_{i=n}^{q}\1_{\mathcal{U}_n^c}\circ \sigma^i\right) \right|\cr
&=\E\left(\prod_{i=2n}^{q}\1_{\mathcal{U}_n^c}\circ \sigma^i \cdot \left( 1-\prod_{i=n}^{2n-1}\1_{\mathcal{U}_n^c}\circ \sigma^i \right) \right)  \cr
&\leq \E\left( 1-\prod_{i=n}^{2n-1}\1_{\mathcal{U}_n^c}\circ \sigma^i  \right) \leq \sum_{i=n}^{2n-1}\E\left( \1_{\mathcal{U}_n}\circ \sigma^i \right)=n\epsilon_n.
\end{aligned}
$$

\item[Proof of (\ref{eq_same_lemma_4})] Let $n,k\in \N$ and $q\in \{ 2n+1 ,\ldots, N-1\}.$ We have
$$
\begin{aligned}
& \left|\epsilon_n\E\left( \prod_{i=2n}^{q}\1_{\mathcal{U}_n^c}\circ \sigma^i  \right) -\E\left( \1_{\mathcal{U}_n} \prod_{i=2n}^{q}\1_{\mathcal{U}_n^c}\circ \sigma^i  \right)\right|\cr
& \leq \psi_n \epsilon_n \E\left( \prod_{i=2n}^{q}\1_{\mathcal{U}_n^c}\circ \sigma^i  \right) \leq \psi_n \epsilon_n,
\end{aligned}
$$
because of the condition of $\psi$-mixing.

\end{proof}

We can now bound $S_2(N).$

\begin{lemma} For all $n\in \N,$ we have that $$S_2(N)\leq 4(n+1)(t+1)\epsilon_n+(t+1)\psi_n+ N\epsilon_n^2\sum_{i=0}^{n-1}\psi_i.$$
\end{lemma}

\begin{proof}
Recall that $S_2(N)=\sum_{q=n}^{N-1} p_q(n).$ From (\ref{george_number_eq}) we have that
$$
\begin{aligned}
p_q(n)&=(1-\epsilon_n)^{N-q-1}\left( \epsilon_n\E\left( \prod_{i=1}^{q}\1_{\mathcal{U}_n^c}\circ \sigma^i  \right) -\E\left( \1_{\mathcal{U}_n} \prod_{i=1}^{q}\1_{\mathcal{U}_n^c}\circ \sigma^i  \right)\right)\\
&\leq \epsilon_n\E\left( \prod_{i=1}^{q}\1_{\mathcal{U}_n^c}\circ \sigma^i  \right) -\E\left( \1_{\mathcal{U}_n} \prod_{i=1}^{q}\1_{\mathcal{U}_n^c}\circ \sigma^i  \right)
\end{aligned}
$$
for every $q\in \{n ,\ldots, N-1\}.$
For a fixed $q$ we bound $p_q(n)$ by the sum of two terms: 
$$
\begin{aligned}
&\epsilon_n\E\left( \prod_{i=1}^{q}\1_{\mathcal{U}_n^c}\circ \sigma^i  \right) -\E\left( \1_{\mathcal{U}_n} \prod_{i=1}^{q}\1_{\mathcal{U}_n^c}\circ \sigma^i  \right)\cr
&\leq \left|\epsilon_n\E\left( \prod_{i=1}^{q}\1_{\mathcal{U}_n^c}\circ \sigma^i  \right)-\epsilon_n\E\left( \prod_{i=n}^{q}\1_{\mathcal{U}_n^c}\circ \sigma^i  \right)\right|\cr
&\quad +\left|\epsilon_n\E\left( \prod_{i=n}^{q}\1_{\mathcal{U}_n^c}\circ \sigma^i  \right) -\E\left( \1_{\mathcal{U}_n} \prod_{i=n}^{q}\1_{\mathcal{U}_n^c}\circ \sigma^i  \right)\right|=: S_{21}(N,q) +S_{22}(N,q).
\end{aligned}
$$
Notice that we have used (\ref{cond10032015:1}) to obtain that $$\E\left( \1_{\mathcal{U}_n} \prod_{i=1}^{q}\1_{\mathcal{U}_n^c}\circ \sigma^i  \right)=\E\left( \1_{\mathcal{U}_n} \prod_{i=n}^{q}\1_{\mathcal{U}_n^c}\circ \sigma^i  \right).$$
Our goal now is to bound $S_{21}(N,q)$ and $S_{22}(N,q).$ For the first term we have the simple inequality $S_{21}(N,q)\leq n\epsilon_{n}^2,$ because
$$
\begin{aligned}
S_{21}(N,q)&\leq \epsilon_{n}\E\left(\left|\prod_{i=1}^{n-1}\1_{\mathcal{U}_n^c}\circ \sigma^i-1 \right|\cdot \prod_{i=n}^{q}\1_{\mathcal{U}_n^c}\circ \sigma^i  \right) \cr
&\leq \epsilon_{n}\E\left(1-\prod_{i=1}^{n-1}\1_{\mathcal{U}_n^c}\circ \sigma^i \right) \leq\epsilon_{n}\E\left(\sum_{i=1}^{n-1}\1_{\mathcal{U}_n}\circ \sigma^i \right) =n\epsilon_{n}^2 .
\end{aligned}
$$
To bound $S_{22}(N,q)$ we use Lemma \ref{lem_10032015:4}.
Suppose that $q>2n,$ then 
$$
\begin{aligned}
S_{22}(N,q)&=\left|\epsilon_n\E\left( \prod_{i=n}^{q}\1_{\mathcal{U}_n^c}\circ \sigma^i  \right) -\E\left( \1_{\mathcal{U}_n} \prod_{i=n}^{q}\1_{\mathcal{U}_n^c}\circ \sigma^i  \right)\right| \cr
&\leq \left|\epsilon_n\E\left( \prod_{i=2n}^{q}\1_{\mathcal{U}_n^c}\circ \sigma^i  \right) -\E\left( \1_{\mathcal{U}_n} \prod_{i=2n}^{q}\1_{\mathcal{U}_n^c}\circ \sigma^i  \right)\right| \cr
&\quad + \epsilon_n\left|\E\left( \prod_{i=n}^{q}\1_{\mathcal{U}_n^c}\circ \sigma^i  \right) -\E\left(\prod_{i=2n}^{q}\1_{\mathcal{U}_n^c}\circ \sigma^i  \right)\right| \cr
&\quad + \left|\E\left( \1_{\mathcal{U}_n} \prod_{i=2n}^{q}\1_{\mathcal{U}_n^c}\circ \sigma^i  \right) -\E\left( \1_{\mathcal{U}_n} \prod_{i=n}^{q}\1_{\mathcal{U}_n^c}\circ \sigma^i  \right)\right|\cr
&=: I_1(N,q)+I_2(N,q)+I_3(N,q). 
\end{aligned}
$$
We have that 
$I_1(N,q)\leq \epsilon_n \psi_n$ by (\ref{eq_same_lemma_4}),
$I_2(N,q)\leq n\epsilon_n^2$ by  (\ref{eq_same_lemma_3}) and 
$I_3(N,q)\leq \epsilon_n^2\sum_{i=0}^{n-1}\psi_i+n\epsilon_n^2$ by  (\ref{eq_same_lemma_2}). In the case $n\leq q\leq 2n$ we can use that 
$S_{22}(N,q)\leq  \epsilon_n.$ Finally,
$$
\begin{aligned}
S_2(N)&\leq \sum_{q=n}^N \left(S_{21}(N,q)+S_{22}(N,q)\right)\cr
&= \sum_{q=n}^N S_{21}(N,q) +\sum_{q=n}^{2n} S_{22}(N,q) + \sum_{q=2n+1}^N S_{22}(N,q)\cr 
&\leq N n\epsilon_n^2+ (n+1)\epsilon_n+N(I_1(N,q)+I_2(N,q)+I_3(N,q))\cr
&\leq \epsilon_n (n(t+1)+(n+1))+ (t+1) \psi_n+ 2(t+1) n\epsilon_n+ N \epsilon_n^2\sum_{i=0}^{n-1}\psi_i,
\end{aligned}
$$
which concludes the proof.
\end{proof}

\subsection{First corollary}\label{sub_2_25_4_16}

The proof of the corollary is very similar to the one of Proposition \ref{prop1_10032015}, but we need to modify some details. To complete the proof we require the following lemma, where $n\in\N$ and $q>n.$

\begin{lemma}\label{lem2:21Mar}
If $\eta_n=n$ then 
$$
\begin{aligned}
&\left| \E\left(\1_{\mathcal{U}_n}\prod_{i=n}^q \1_{\mathcal{U}_n^c}\circ \sigma^i \right)- \E\left(\1_{\mathcal{U}_n}\prod_{i=\eta_n}^q \1_{\mathcal{U}_n^c}\circ \sigma^i \right)\right| =0
\end{aligned}
$$
and if $\eta_n<n$ then
$$
\left| \E\left(\1_{\mathcal{U}_n}\prod_{i=n}^q \1_{\mathcal{U}_n^c}\circ \sigma^i \right)- \E\left(\1_{\mathcal{U}_n}\prod_{i=\eta_n}^q \1_{\mathcal{U}_n^c}\circ \sigma^i \right)\right| \leq n\epsilon_n \epsilon_{\left\lfloor \eta_n/2 \right\rfloor}( 1+ \psi_{\left\lfloor \eta_n/2 \right\rfloor}).
$$
\end{lemma}

\begin{proof}
\begin{figure}[!]
\centerline{
 \fbox{
\begin{tikzpicture}
\filldraw[draw=black,fill=black!20] (0,0) rectangle (10,1);
\filldraw[draw=black,fill=black!20] (4,-1) rectangle (14,0);
\draw[dashed] (2,1) --(2,0);
\node [below] at (2,-1.5) {$i$};
\node [below] at (9,-2) {$n$};
\node [above] at (5,2.5) {$n$};
\node [below] at (9,0) {$\mathcal U_n=[x^1]_n\sqcup\cdots\sqcup[x^{m_n}]_n$};
\node [above] at (1,1.5) {$[i/2]$};
\draw[] (14,0.5) --(10,0.5);
\node [above] at (12,0.5) {$\sigma^i$};
\node [below] at (1,0.8) {$\mathcal U_{[i/2]}$};
\draw[decorate, decoration={brace, mirror, raise=6pt}]
    (0,-1) --(4,-1);
\draw[decorate, decoration={brace, mirror, raise=6pt}]
    (10,2) --(0,2);
\draw[decorate, decoration={brace, mirror, raise=6pt}]    
     (2,1)--(0,1);
\draw[decorate, decoration={brace, mirror, raise=6pt}]    
     (4,-1.5) --(14,-1.5);
\end{tikzpicture}
}}
\caption{Proof of the inequality (\ref{26_Jun_2015}).}\label{help_26_Jun}
\end{figure}
The case $\eta_n=n$ is trivial, so suppose that $\eta_n<n.$ It is clear that
$$
\begin{aligned}
&\left| \E\left(\1_{\mathcal{U}_n}\prod_{i=n}^q \1_{\mathcal{U}_n^c}\circ \sigma^i \right)- \E\left(\1_{\mathcal{U}_n}\prod_{i=\eta_n}^q \1_{\mathcal{U}_n^c}\circ \sigma^i \right)\right| \leq \sum_{i=\eta_n}^{n-1}\E\left(\1_{\mathcal{U}_n}\cdot \1_{\mathcal{U}_n}\circ\sigma^i\right).
\end{aligned}
$$

In Figure \ref{help_26_Jun} we have chosen $i\in \{ \eta_n ,\ldots, n-1\} $ and we represent the set $\mathcal{U}_n=[x^1]_n\sqcup\cdots\sqcup[x^{m_n}]_n$ for some $x^1,\ldots,x^{m_n}\in \mathcal X,$ $m_n\in\N$ and the set $\sigma^i\mathcal{U}_n.$ We have also draw a representation of the set $\mathcal{U}_{[i/2]}.$ We can see that the action of the shift moved the rectangle at the bottom to the left, the result is the rectangle that we draw at the top. It is clear that
$$
\E\left(\1_{\mathcal{U}_n}\cdot \1_{\mathcal{U}_n}\circ\sigma^i\right)\leq \E\left(\1_{\mathcal{U}_n}\cdot \1_{\mathcal{U}_{[i/2]}}\circ\sigma^i\right)=\mu\left(\mathcal{U}_n\cap \sigma^{-i}\mathcal{U}_{[i/2]}\right). 
$$
From the $\psi$-mixing condition of the measure $\mu$ we obtain that 
$$
|\mu(\mathcal{U}_n\cap \sigma^{-i}\mathcal{U}_{[i/2]})-\epsilon_n \epsilon_{[i/2]} |\leq \psi_{[i/2]} \epsilon_{[i/2]}\epsilon_n,
$$ 
and therefore 
\begin{equation}\label{26_Jun_2015}
\begin{aligned}
\sum_{i=\eta_n}^{n-1}\E\left(\1_{\mathcal{U}_n}\cdot \1_{\mathcal{U}_n}\circ\sigma^i\right) & \leq \sum_{i=\eta_n}^{n-1} \left(\psi_{[i/2]}
\epsilon_{[i/2]} \epsilon_{n} +\epsilon_n \epsilon_{[i/2]} \right)\\
& \leq n\epsilon_n \epsilon_{\left\lfloor \eta_n/2 \right\rfloor}(1+\psi_{\left\lfloor \eta_n/2 \right\rfloor}).
\end{aligned}
\end{equation}
\end{proof}

The proof of Corollary \ref{prop2_10032015} comes from the observation that for $q\in \{ n ,\ldots, N-1\} $ we have

$$
\begin{aligned}
&\left|\epsilon_n\E\left( \prod_{i=1}^{q}\1_{\mathcal{U}_n^c}\circ \sigma^i  \right) -\E\left( \1_{\mathcal{U}_n} \prod_{i=1}^{q}\1_{\mathcal{U}_n^c}\circ \sigma^i  \right)\right| \cr
& \leq \left|\epsilon_n\E\left( \prod_{i=1}^{q}\1_{\mathcal{U}_n^c}\circ \sigma^i  \right)-\epsilon_n\E\left( \prod_{i=n}^{q}\1_{\mathcal{U}_n^c}\circ \sigma^i  \right)\right|\cr
&\quad +\left|\epsilon_n\E\left( \prod_{i=n}^{q}\1_{\mathcal{U}_n^c}\circ \sigma^i  \right) -\E\left( \1_{\mathcal{U}_n} \prod_{i=n}^{q}\1_{\mathcal{U}_n^c}\circ \sigma^i  \right)\right|\cr
&\quad +\left| \E\left(\1_{\mathcal{U}_n}\prod_{i=n}^q \1_{\mathcal{U}_n^c}\circ \sigma^i \right)- \E\left(\1_{\mathcal{U}_n}\prod_{i=\eta_n}^q \1_{\mathcal{U}_n^c}\circ \sigma^i \right)\right| \cr
&:= S_{21}(N,q) +S_{22}(N,q)+S_{\text{extra}}(N,q).
\end{aligned}
$$
We can use Lemma \ref{lem2:21Mar} to bound $S_{\text{extra}}(N,q)$ and the proof of Corollary \ref{prop2_10032015} follows directly from the proof of Proposition \ref{prop1_10032015}.

\subsection{Theorem \ref{teo:29062015} }\label{sub_3_25_4_16}

We can use the same notation and definitions used in the previous proofs and write $$\mu(\tau_n>N)-(1-\epsilon_n)^N=S_1(N)+S_2(N),$$
where $$S_2(N)\leq\sum_{q=n}^{N-1} S_{21}(N,q)+S_{22}(N,q)+S_{\text{extra}}(N,q) \mbox{ and}$$
$$S_{\text{extra}}(N,q) :=\left| \E\left(\1_{\mathcal{U}_n}\prod_{i=n}^q \1_{\mathcal{U}_n^c}\circ \sigma^i \right)- \E\left(\1_{\mathcal{U}_n}\prod_{i=\eta_n}^q \1_{\mathcal{U}_n^c}\circ \sigma^i \right)\right| \leq \sum_{i=\eta_n}^{2n-1}\E\left(\1_{\mathcal{U}_n}\cdot \1_{\mathcal{U}_n}\circ\sigma^i\right).$$
Notice that the case \ref{c1_250216}. is exactly the condition required in Corollary \ref{prop2_10032015}. Therefore, we only need to prove the theorem in the cases \ref{c2_250216}. and \ref{c3_250216}. The unique difference in these cases and the one of First corollary is that we need to find a new bound for $\E\left(\1_{\mathcal{U}_n}\cdot \1_{\mathcal{U}_n}\circ\sigma^i\right)$ when $i\in \{ \eta_n ,\ldots, 2n\}.$ 

\begin{lemma} 
If $(\mathcal{X},\B_{\mathcal{X}},\mu,\sigma,\{\mathcal{U}_n\},\{X_n\})$ is a $M_a$-system where $\mathcal{U}_n$ is in the sigma algebra generated by $\prod_{-n}^{n-1}\{1,\ldots,a\}$ for every $n\geq 1.$  Then, for every $n\geq 1$
$$
\begin{aligned}
&\E\left(\1_{\mathcal{U}_n}\cdot \1_{\mathcal{U}_n}\circ\sigma^i\right) \\
&\leq\begin{cases}
\mu(\mathcal{U}_n)\mu\left(\mathcal{U}_{k}\right)+ \psi_{\triangle} \epsilon_k \epsilon_n &\mbox{ if }\eta_n=n+k+\triangle \mbox{ with } k+\triangle \geq 1,\\
\mu(\mathcal V_n) \mu(\mathcal U_n)+\psi_{\left\lfloor \eta_n/2 \right\rfloor} \mu(\mathcal{V}_{n}) \epsilon_n &\mbox{ if } \eta_n<n+1, 
\end{cases}
\end{aligned}
$$
for all $i\geq  \eta_n.$
\end{lemma}

\begin{proof}
Let us fix $n\in \N.$ We have two cases: $\eta_n=n+k+\triangle \mbox{ with } k+\triangle \geq 1$ or   
$\eta_n<n+1.$ Suppose first that $\eta_n=n+k+\triangle \mbox{ with } k+\triangle \geq 1$ and that $i\in \{ \eta_n ,\ldots, 2n\},$ then $$\E\left(\1_{\mathcal{U}_n}\cdot \1_{\mathcal{U}_n}\circ\sigma^i\right)\leq \E\left(\1_{\mathcal{U}_k}\cdot \1_{\mathcal{U}_n}\circ\sigma^i\right)=\mu\left( \mathcal{U}_k\cap \sigma^{-i}\mathcal{U}_n\right)$$
and 
\begin{equation}\label{eq1:30062015}
\begin{aligned}
| \mu\left( \mathcal{U}_k\cap \sigma^{-i}\mathcal{U}_n\right)-\mu(\mathcal{U}_k) \mu(\mathcal{U}_n) |\leq \psi_{\triangle} \epsilon_k\epsilon_n.
\end{aligned}
\end{equation}
Therefore $$\E\left(\1_{\mathcal{U}_n}\cdot \1_{\mathcal{U}_n}\circ\sigma^i\right)\leq \mu(\mathcal{U}_k) \mu(\mathcal{U}_n)+\psi_{\triangle} \epsilon_k \epsilon_n.$$
In Figure \ref{segunda_help_27_Jun} we represented the sets $\mathcal{U}_n\cap \sigma^{-i}\mathcal{U}_n$ and $\mathcal{U}_k\cap \sigma^{-i}\mathcal{U}_n.$ The light grey rectangle at the top represents the set $\mathcal{U}_n$ and the one at the bottom the set $\sigma^{-i}\mathcal{U}_n.$ The darker grey rectangle represents the set $\mathcal{U}_k\supset \mathcal{U}_n.$ The ``gap'' $\triangle$ between the coordinates fixed by the sets $\mathcal{U}_k$ and  $\sigma^{-i}\mathcal{U}_n$ allows to use the $\psi$-mixing condition of the measure $\mu$ to conclude inequality (\ref{eq1:30062015}).\\
\begin{figure}[!]
\centerline{
\fbox{
\begin{tikzpicture}
\filldraw[draw=black,fill=black!20] (-1,0) rectangle (9,1);
\filldraw[draw=black,fill=black!20] (6,-1) rectangle (16,0);
\filldraw[draw=black,fill=black!40] (3,0) rectangle (5,1);
\draw[dashed] (4,1) --(4,0);
\node [below] at (2.5,-2.5) {$i$};
\node [below] at (11,-3) {$2n$};
\node [above] at (4,3.5) {$2n$};
\node [above] at (4,2.5) {$2k$};
\node [below] at (5.5,-1.5) {$\triangle$};
\node[right] (5,-4) {}
   edge[] (11,-0.5);
\node[above] (10,4) {}
   edge[] (4,0.5);      
\node [below] at (5,-4) {$\mathcal U_n=[x^1_{-n},\ldots,x_0^1,\ldots, x^1_{n-1}]_{-n}^{n}\sqcup\ldots \sqcup[x^{m_n}_{-n},\ldots,x^{m_n}_0,\ldots, x^{m_n}_{n-1}]_{-n}^{n}$};
\node [above] at (10,4) {$\mathcal U_k=[x^1_{-k},\ldots,x^1_0,\ldots, x^1_{k-1}]_{-k}^{k}\sqcup\ldots\sqcup[x^{m_k}_{-k},\ldots,x^{m_k}_0,\ldots, x^{m_k}_{k-1}]_{-k}^{k}$};
\draw[] (16,0.5) --(9,0.5);
\node [above] at (12.5,0.5) {$\sigma^i$};
\draw[decorate, decoration={brace, mirror, raise=6pt}]
    (-1,-2) --(6,-2);
\draw[decorate, decoration={brace, mirror, raise=6pt}]
    (9,3) --(-1,3);
\draw[decorate, decoration={brace, mirror, raise=6pt}]
    (5,2) --(3,2);
\draw[decorate, decoration={brace, mirror, raise=6pt}]    
     (6,-2.5) --(16,-2.5);
\draw[decorate, decoration={brace, mirror, raise=6pt}]    
     (5,-1) --(6,-1);\end{tikzpicture}
}}
\caption{Proof of the inequality (\ref{eq1:30062015}).}\label{segunda_help_27_Jun}
\end{figure}

Suppose now that $\eta_n<n+1$ and that $\mathcal V_n$ has coordinates fixed only in $\{-n,\ldots,-n+\left\lfloor \eta_n/2 \right\rfloor\}$ (the case that $\mathcal V_n$ has coordinates fixed only in $\{n-\left\lfloor \eta_n/2 \right\rfloor,\ldots,n-1\}$ is similar), then $$\E\left(\1_{\mathcal{U}_n}\cdot \1_{\mathcal{U}_n}\circ\sigma^i\right)\leq \E\left(\1_{\mathcal{V}_n}\cdot \1_{\mathcal{U}_n}\circ\sigma^i\right)=\mu\left( \mathcal{V}_n\cap \sigma^{-i}\mathcal{U}_n\right)$$
and 
\begin{equation}\label{eq2:30062015}
\begin{aligned}
| \mu\left( \mathcal{V}_n\cap \sigma^{-i}\mathcal{U}_n\right)-\mu(\mathcal V_n) \mu(\mathcal{U}_n) |\leq \psi_{\left\lfloor \eta_n/2 \right\rfloor} \mu(\mathcal V_{n}) \epsilon_n.
\end{aligned}
\end{equation}
Therefore $$\E\left(\1_{\mathcal{U}_n}\cdot \1_{\mathcal{U}_n}\circ\sigma^i\right)\leq \mu(\mathcal V_n) \mu(\mathcal U_n)+\psi_{\left\lfloor \eta_n/2 \right\rfloor} \mu(\mathcal{V}_{n})\epsilon_n.$$

The gap $\left\lfloor \eta_n/2 \right\rfloor$ between the coordinates fixed by the sets $\mathcal{V}_n$ and  $\sigma^{-i}\mathcal{U}_n$ allows to use the $\psi$-mixing condition of the measure $\mu$ to conclude inequality (\ref{eq2:30062015}).

\end{proof}

\section*{Acknowledgments}
This work constitutes part of my Ph.D., which was supported by CONICYT. I would like to especially thank my supervisor Mark Pollicott for bringing my attention to the content of this work, many useful discussions and reading draft versions. I would also like to thank Thomas Jordan, Ian Melbourne and the anonymous reviewer for suggesting the final version of this work.


\begin{thebibliography}{0}

\bibitem{Ab2001}
M. Abadi, Exponential approximation for hitting times in mixing stochastic processes, {\it Math. Phys. Electron. J.} {\bf 7} (2001).

\bibitem{AbGal2001}
M. Abadi and A. Galves, Inequalities for the occurrence times of rare events in mixing processes. The state of the art, {\it Markov Proc. Relat. Fields} {\bf 7} (2001) 97--112.

\bibitem{AbSa2011}
M. Abadi and B. Saussol, Hitting and returning into rare events for all alpha-mixing processes, {\it Stoch. Proc. Appl.} {\bf 121} (2011) 314--323.

\bibitem{AandB1992}
D.J. Aldous and M. Brown, Inequalities for rare events in time-reversible Markov chains I, {\it Stochastic Inequalities IMS Lecture Notes} {\bf 22} (1992) 1--16.

\bibitem{AandB1993}
D.J. Aldous and M. Brown, Inequalities for rare events in time-reversible Markov chains II, {\it Stoch. Proc. Appl.} {\bf 44} (1993) 15--25.

\bibitem{Rufus}
R. Bowen, {\it Equilibrium states and the Ergodic theory of Anosov diffeomorphisms}, (Lecture Notes in Math. 470, Springer, Berlin, 1975).

\bibitem{Bruin_returntime}
H. Bruin and S. Vaienti, Return time statistics for unimodal maps, {\it Fundam. Math.} {\bf 176} (2003) 77--94.

\bibitem{ChCo2013}
J-R. Chazottes and P. Collet, Poisson approximation for the number of visits to balls in nonuniformly hyperbolic dynamical systems, {\it Ergodic Theory Dynam. Systems} {\bf 33} (2013).

\bibitem{Collet2006}
J-R. Chazottes and Z. Coelho and P. Collet, Poisson processes for subsystems of finite type in symbolic dynamics, {\it Stoch. Dyn.} {\bf 9} (2009).

\bibitem{MR1776758}
Z. Coelho, Asymptotic laws for symbolic dynamical systems, {\it Topics in symbolic dynamics and applications (Temuco, 1997), London Math. Soc. Lecture Note Ser., Cambridge Univ. Press, Cambridge} {\bf 279} (2000).

\bibitem{CEbook}
P. Collet and J-P. Eckmann, {\it Concepts and results in chaotical dynamics},  (Springer, 2006).

\bibitem{GS95}
P. Collet and A. Galves, Asymptotic distribution of entrance times for expanding maps of the interval, {\it Dynamical systems and applications} {\bf WSSIAA 4} (1995) 139--152.

\bibitem{GS93}
P. Collet and A. Galves, Statistics of close visits to the indifferent fixed point of an interval map, {\it J. Stat. Phys.} {\bf 72} (1993) 459--478.

\bibitem{Doeblin40}
W. D{\H o}eblin, Remarques sur la th{\'e}orie m{\'e}trique des fractions continues, {\it Compositio Math.} {\bf 7} (1940) 353--371.

\bibitem{Dolgopyat2004}
D. Dolgopyat, Limit theorems for partially hyperbolic systems, {\it Trans. Amer. Math. Soc.} {\bf 356} (2004) 1637--1689.

\bibitem{GS97}
A. Galves and B. Schmitt, Inequalities for hitting times in mixing dynamical systems, {\it Random Computat. Dynam.} {\bf 5} (1997) 337--348.


\bibitem{MR3170620}
N.T.A. Haydn, Entry and return times distribution, {\it Dyn. Syst.} {\bf 28} (2013) 333--353.

\bibitem{MR2165587}
N.T.A. Haydn and Y. Lacroix and S. Vaienti, Hitting and return times in ergodic dynamical systems, {\it Ann. Probab.} {\bf 33} (2005) 2043--2050.

\bibitem{Haydn_2014}
N.T.A.Haydn and Y.Psiloyenis, Return times distribution for Markov towers with decay of correlations, {\it Nonlinearity} {\bf 27} (2014) 1323--1349.

\bibitem{Hirata}
M. Hirata, Poisson law for Axiom A diffeomorphisms, {\it Ergodic Theory Dynam. Systems} {\bf 13} (1993) 533--556.
	
\bibitem{MH95}
M. Hirata, Poisson law for the dynamical systems with ``self-mixing'' conditions, {\it Dynamical systems and Chaos} {\bf 1} (1995) 87--96.

\bibitem{lindmarcus}
D. Lind and B. Marcus, {\it An introduction to symbolic dynamics and coding}, 2nd edn. (Cambridge University Press, 1995).

\bibitem{Pitskel}
B. Pitskel, Poisson limit law for Markov chains, {\it Ergodic Theory Dynam. Systems} {\bf 11} (1991) 501--513.

\bibitem{MR3101254}
Y. Kifer, Nonconventional Poisson limit theorems, {\it Israel J. Math.} {\bf 195} (2013) 373--392.				

\bibitem{Zhang}
X. Zhang, Note on limit distribution of normalized return times and escape rate, {\it Stoch. Dyn.} {\bf 16} (2016).

\end{thebibliography}
\end{document}